\documentclass[12pt,a4paper,reqno]{amsart}

\usepackage[margin=2.8cm,footskip=1cm]{geometry}
\usepackage[utf8]{inputenc}
\usepackage[T1]{fontenc}
\usepackage[english]{babel}
\usepackage{amsmath}
\usepackage{amsfonts}
\usepackage{amssymb}
\usepackage{amsthm}
\usepackage{ucs}
\usepackage{graphicx}
\usepackage{xcolor}
\usepackage{dsfont}
\usepackage{tikz}
\usepackage{wasysym}
\usepackage{physics}
\usepackage{mathtools}
\usepackage{xifthen}
\usepackage[textwidth=2.5cm,textsize=tiny]{todonotes}
\usepackage{enumitem}
\usepackage{stmaryrd}
\usepackage{constants}
\usepackage[ruled]{algorithm2e}
\usepackage{tikz,pgfplots}
\usepackage[hidelinks]{hyperref}
\usepackage[font={small}]{caption}

\renewcommand{\norm}[1]{\|#1\|}

\newcommand{\lrangle}[1]{\langle #1 \rangle}

\newcommand{\interior}{\mathrm{\scriptscriptstyle int}}
\newcommand{\exterior}{\mathrm{\scriptscriptstyle ext}}

\newcommand{\R}{\mathbb{R}}
\newcommand{\Rd}{\mathbb{R}^d}
\newcommand{\Zd}{\mathbb{Z}^d}

\newcommand{\rateCo}{c_{\mathrm{co}}}
\newcommand{\rateMix}{c_{\mathrm{mix}}}
\newcommand{\cstMix}{C_{\mathrm{mix}}}
\newcommand{\cstEquivDist}{c_{E}}

\newcommand{\bbS}{\mathbb{S}}
\newcommand{\bbZ}{\mathbb{Z}}
\newcommand{\bbR}{\mathbb{R}}
\newcommand{\bbN}{\mathbb{N}}
\newcommand{\bbT}{\mathbb{T}}

\newcommand{\calC}{\mathcal{C}}
\newcommand{\calF}{\mathcal{F}}
\newcommand{\calT}{\mathcal{T}}
\newcommand{\calU}{\mathcal{U}}
\newcommand{\calW}{\mathcal{W}}

\newcommand{\rmd}{\mathrm{d}}

\newcommand{\given}{\,|\,}

\newcommand{\CG}{\mathrm{CG}}

\newcommand{\cone}{\mathcal{Y}}

\newcommand{\clusterSet}{\mathrm{cl}}

\theoremstyle{plain}
\newtheorem{theorem}{Theorem}[section]
\newtheorem{lemma}[theorem]{Lemma}

\newtheorem{corollary}[theorem]{Corollary}

\newtheorem{remark}{Remark}[section]

\theoremstyle{definition}
\newtheorem{definition}{Definition}[section]
\newtheorem{obs}{Observation}

\author{S\'{e}bastien Ott}
\address{Dipartimento di Matematica e Fisica, Univ. Roma Tre, 00146 Roma, Italy}
\email{ott.sebast@gmail.com}

\date{\today}

\title[Existence of the I.C.L. in percolation models]{Existence and properties of connections decay rate for high temperature percolation models}

\begin{document}
	
\begin{abstract}
	We consider generic finite range percolation models on \(\Zd\) under a high temperature assumption (exponential decay of connection probabilities and exponential ratio weak mixing). We prove that the rate of decay of point-to-point connections exists in every directions and show that it naturally extends to a norm on \(\Rd\). This result is the base input to obtain fine understanding of the high temperature phase and is usually proven using correlation inequalities (such as FKG). The present work makes no use of such model specific properties.
\end{abstract}

\maketitle

%%%%%%%%%%%%%%%%%%%%%%%%%%%%%%%%%%%%%%%%%%%%%%%%%%%%%%%%%%%%%%%%%%%%%%%%%%%%%%%%%%%%%%%%%%%%%%%%%%%
\section{Introduction and results}

\subsection{Decay rate of connections}
Let \(P\) denote an edge-percolation measure on \(\Zd\). The central object of our investigation is(are) the rate(s) of exponential decay for point-to-point connection probabilities (two point functions):
\begin{definition}[Inverse correlation length]
	\label{def:ICL}
	Let \(s\in\bbS^{d-1}\). The point-to-point decay rates are
	\begin{equation}
	\begin{gathered}
	\overline{\nu}(s) = \limsup_{n\to\infty} -\frac{1}{n} \log P(0\leftrightarrow n s),\\
	\underline{\nu}(s) = \liminf_{n\to\infty} -\frac{1}{n} \log P(0\leftrightarrow n s).
	\end{gathered}
	\end{equation}
\end{definition}

\subsection{Motivation}

The main motivation of this work comes from the (supposed) \emph{universal behaviour} of two point functions in high temperature systems: they should decay exponentially with a well-defined rate and the pre-factor to this decay should be the one predicted by the Ornstein-Zernike theory~\cite{Ornstein+Zernike-1914,Zernike-1916}. See~\cite{Ott+Velenik-2019} for a review on this topic.

On the one hand some fairly satisfactory universal statements are available in perturbative regimes (\emph{very high temperature} regime), see~\cite{Abraham+Kunz-1977}. In the other hand, a non-perturbative approach (giving statements about the whole high temperature regime) has been developed over the past decades, proving the expected behaviour in various \emph{specific} models:~\cite{Abraham+Chayes+Chayes-1984,Chayes+Chayes-1986,Ioffe-1998,Campanino+Ioffe-2002,Campanino+Ioffe+Velenik-2003,Campanino+Ioffe+Velenik-2008,Ott-2020}. A recurrent ingredient in the proofs being the presence of correlation inequalities.

The latest non-perturbative approaches (mainly~\cite{Campanino+Ioffe+Velenik-2008} combined with refinements from~\cite{Ott+Velenik-2018} and~\cite{Ott-2020}) seem to be robust enough to tackle the problem (with some work...) for \emph{generic percolation models} under a high temperature assumption and conditionally on the decay rate existence as well as some of its properties validity. This latter condition is usually where correlation inequalities are crucially needed.

To give an idea of the problem, let us consider some translation invariant percolation model \(P\). When \(P\) satisfies the FKG inequality, one has \(P(0\leftrightarrow x+y)\geq P(0\leftrightarrow x\leftrightarrow x+y)\geq P(0\leftrightarrow x)P(0\leftrightarrow y)\). The equality \(\overline{\nu}=\underline{\nu}\equiv\nu\) is then easy consequence of Fekete's Lemma. One can further extend \(\nu\) by positive homogeneity. The above inequality directly implies that \(\nu\) satisfies the triangle inequality.

The main problem is that ``satisfying FKG inequality'' (or any other) is a very model specific property (which fails for some arbitrarily small perturbation of -for example- FK percolation) while \(\overline{\nu}=\underline{\nu}\) and \(\nu\) being a norm should be a generic property of high temperature models (which is a condition insensitive to sufficiently small perturbations). We therefore introduce a suitable notion of high temperature phase for percolation models and prove that the wanted properties of \(\nu\) hold for any model in this phase. To the best of the author's knowledge, this is the first non-perturbative proof of this type of result not relying on correlation inequalities.

\subsection{Results}

Our main result is (see Section~\ref{sec:def_notation} for missing definitions):
\begin{theorem}
	\label{thm:existence}
	Let \(E\subset \big\{\{i,j\}\subset \Zd\big\}\) be finite range, irreducible, invariant under translations. Let \(P\) be a  percolation measure on \(E\). Suppose that
	\begin{itemize}
		\item \(P\) is invariant under translations,
		\item \(P\) has the insertion tolerance property (Definition~\ref{def:insertion_tolerance}) with constant \(\theta>0\),
		\item \(P\) satisfies the exponential ratio weak mixing property (Definition~\ref{def:ratio_mix}) with rate \(\rateMix>0\) and constant \(\cstMix<\infty\) for the set of local connection events,
		\item there exists \(\rateCo>0\) such that \( P(0\leftrightarrow \Lambda_n^c)\leq e^{-\rateCo n}\) for any \(n\) large enough.
	\end{itemize}
	
	Then, for any \(s\in\bbS^{d-1}\),
	\begin{equation}
	\overline{\nu}(s) = \underline{\nu}(s) \equiv \nu(s).
	\end{equation}Moreover, the extension of \(\nu\) by positive homogeneity of order one defines a norm on \(\Rd\).
\end{theorem}

\begin{remark}
	\label{rem:ratio_weak_mix}
	The ratio weak mixing condition demanded can look less natural and more stringent than the weak mixing property (not ratio). However, it has been shown, see~\cite[Theorem 3.3]{Alexander-1998}, that in many cases the two are equivalent. In particular, if \(P(\omega)\propto \prod_{C\in\clusterSet(\omega)} f(C)\) (formally, the R.H.S. being infinite, \(\clusterSet\) denotes the set of connected components), the assumption \(P(0\leftrightarrow \Lambda_N^c\given \calF_{E\setminus E(\Lambda_N)})\leq e^{-cN}\) implies that the model has exponentially bounded controlling regions in the sense of~\cite{Alexander-1998}.
	
	Moreover, modulo straightforward changes in the proofs, one can replace the exponential mixing by any power law mixing with power \(>d\). But this type of mixing can generally be enhanced to exponential (see for example the discussion on mixing in~\cite{Martinelli-1999}).
\end{remark}
\begin{remark}
	\label{rem:ins_tol}
	The insertion tolerance property excludes degeneracies occurring in hard-core models. Moreover, it gives lower bounds on local connections implying for example that the decay rates of Definition~\ref{def:ICL} are in \((\epsilon,\epsilon^{-1})\) for some \(\epsilon>0\) (non-degenerate).
\end{remark}
\begin{remark}
	\label{rem:sym}
	\(\nu\) obviously inherit additional symmetries of \(P\).
\end{remark}

\subsection{Strategy of the proof}

The idea of the proof goes as follows: one expects that existence of \(\nu\) and \(\nu\) being a norm is closely related to some form of sub-additivity. The latter property can be recovered from mixing \emph{if typical clusters realizing connections are somehow directed}. We thus introduce various notions of ``directed connections'' for which we prove existence of an asymptotic decay rate. We then show that all these rates are equal and define a norm \(\tilde{\nu}\). To relate the obtained ``directed rate'' to the ``real rates'', we do a small detour: we introduce point to hyperplanes decay rates and their directed version. Showing that these two agree is much easier than for point-to-point connections and is done using a suitable coarse-graining argument. We then relate directed point-to-point to directed point-to-hyperplane via convex duality (approximately: directed point-to-hyperplane connections in a direction \(s\) are realized by a directed point-to-point connection in an optimal direction \(s'\)). Finally, we relate (non directed) point-to-point connections to (non directed) point-to-hyperplane connections via another coarse-graining argument.

\section{Definitions and notations}
\label{sec:def_notation}
\subsection{General notations}
Denote \(\norm{\ }\) the Euclidean norm on \(\Rd\) and \(\rmd\) the associated distance. Write \(\bbS^{d-1}\) the unit sphere for \(\norm{\ }\). \(\lrangle{\ ,\ }\) will denote the scalar product. \(s\) will always be an element of \(\bbS^{d-1}\). For a (possibly asymmetric) norm \(\mu:\Rd\to\R_+\), define the unite ball of \(\mu\) and its polar set (``Wulff shape'')
\begin{equation*}
	\calU_{\mu} = \{x\in\Rd:\ \mu(x)\leq 1 \},\quad \calW_{\mu} = \bigcap_{s\in\bbS^{d-1}} \{x\in\Rd:\ \lrangle{x,s}\leq \mu(s) \}.
\end{equation*}
For \(A\subset \Rd\) and \(x\in \Rd\), write \(A+x\) the translate of \(A\) by \(x\), \(\partial A\) the boundary of \(A\) and \(\mathring{A}=A\setminus \partial A\) the interior of \(A\).

Denote
\begin{equation*}
	\Lambda_N = [-N,N]^d,\quad \Lambda_N(x) = x+\Lambda_N.
\end{equation*}We also denote \(\Lambda_N\) the intersection of \(\Lambda_N\) with \(\Zd\).

Define the half spaces: for \(s\in\bbS^{d-1}\),
\begin{equation}
\label{eq:def:half_space}
	H_{s} = \{x\in\Rd:\ \lrangle{x,s}\geq 0\},\quad H_{s}(x) = x+H_s.
\end{equation}
Then, for \(\delta\in [0,1]\), define the cones
\begin{equation}
\label{eq:def:cones}
\cone_{s,\delta} = \{x\in\Rd:\ \lrangle{x,s} \geq (1-\delta) \norm{x} \},\quad \cone_{s,\delta}(x) = x+\cone_{s,\delta}.
\end{equation}
\(\delta=0\) is a line and \(\delta =1 \) is the half space \(H_s\).

Also introduce the truncated cones
\begin{equation}
	\label{eq:def:trunc_cones}
	\cone_{s,\delta}^K = \cone_{s,\delta} \setminus H_{s}(Ks),\quad \cone_{s,\delta}^K(x)= x+\cone_{s,\delta}^K.
\end{equation}

For \(x\in\Rd\), we denote \(\mathrm{int}(x)\) the point in \(\Zd\) closest to \(x\), with some fixed breaking of draws respecting symmetries/translations of \(\Zd\). We will often omit \(\mathrm{int}\) from the notation.

We fix a priori some arbitrary total order on \(\Zd\).

We will regularly use the following notation: for \((a_n)_{n\geq 1}\in \bbR^{\bbN}\) a sequence, we denote \(\overline{a} = \limsup_{n\to\infty} a_n\in\bbR\cup \{\pm \infty\}\), \(\underline{a}=\liminf_{n\to\infty} a_n\in\bbR\cup \{\pm \infty\}\). When \(\overline{a}=\underline{a}\), we write the limit \(a\).

A quantity \(f(n)\) is \(o_n(1)\) if \(\lim_{n\to\infty} f(n) = 0\).

\subsection{Percolation}
We consider edge percolation models, in all this work \(E\) will be a subset of \(\big\{\{i,j\}\subset \Zd \big\}\) with the properties:
\begin{itemize}
	\item \emph{Irreducibility}: \((\Zd,E)\) is connected.
	\item \emph{Finite Range}: there exists \(r>0\) such that \(\norm{i-j}\geq r \implies \{i,j\}\notin E\). The smallest such \(r\) is called the \emph{range} of \(E\) and is denoted \(R\equiv R(E)\).
	\item \emph{Translation Invariance}: for any \(e\in E\) and \(x\in\Zd\), \(x+e\in E\).
%	\item \emph{Central Symmetry}: \(E\) is invariant under reflection through \(0\).
\end{itemize}

The graph distance on \((\Zd,E)\) is denoted \(\rmd_E\). As \(E\) is finite range and irreducible, there exists \(\cstEquivDist >0\) such that
\begin{equation*}
	\cstEquivDist^{-1} \rmd(x,y)\leq \rmd_E(x,y)\leq \cstEquivDist \rmd(x,y).
\end{equation*}

For a set \(A\subset \Zd\), denote \(A^c=\Zd\setminus A\), \(\partial^{\interior}A = \{x\in A:\ \exists y\in A^c,\ \{x,y\}\in E  \}\), \(\partial^{\exterior}A = \{x\in A^c:\ \exists y\in A,\ \{x,y\}\in E  \}\). Also define \(E(A) =\{\{x,y\}\in E:\ x,y\in A\}\).

For \(\omega\in\{0,1\}^E\), we systematically identify the \(\{0,1\}\)-valued function and the edge set induced by the set \(\{e:\omega_e=1\}\), the set of open edges. When talking about connectivity properties of \(\omega\), it is assumed that the graph \((\Zd,\omega)\) is considered.

For \(F\subset E\) finite, denote \(\calF_{F} = \{A\subset \{0,1\}^F \} \) and for \(F\subset E\) infinite, denote \(\calF_F\) the sigma algebra generated by the collection \((\calF_{F'})_{F'\subset F \textnormal{ finite}}\). A percolation measure \(P\) is a probability measure such that \((P,\calF_E, E)\) is a probability space. We write \(\{x\leftrightarrow y\}\) for the event that \(x,y\) lie in the same connected component (and \(\{A\leftrightarrow B\}\) for the event that there exists \(x\in A, y\in B\) with \(x\leftrightarrow y\)). We also will write \(\{x\xleftrightarrow[]{F} y\}\) for the event that \(x\) is connected to \(y\) by a path of open edges in \(F\). \(\omega\) will be a random variable with law \(P\).

\subsection{Hypotheses}
One of our hypotheses is a mixing condition, called the \emph{exponential ratio weak mixing property} for connections events:
\begin{definition}[Ratio mixing]
	\label{def:ratio_mix}
	We say that \(P\) has the \emph{ratio weak mixing property} with rate \(c>0\) and constant \(C\geq 0\) if for any sets \(F,F'\subset E\) and events \(A\in\calF_F,B\in\calF_{F'}\) with \(P(A)P(B)>0\),
	\begin{equation}
	\label{eq:def:ratio_mix}
		\Big|1- \frac{P(A\cap B)}{P(A) P(B)}\Big| \leq C\sum_{e\in F, e'\in F'} e^{-c \rmd(e,e')},
	\end{equation}where \(\rmd\) is the Euclidean distance. We say that the property is satisfied for the class \(\calC\subset \calF_E\) if~\eqref{eq:def:ratio_mix} holds whenever, in addition to the hypotheses, \(A,B\in \calC\).
\end{definition}

\begin{definition}[Connexion events]
	The class of \emph{local connection events} is the set of events of the form
	\begin{equation*}
		\{A\xleftrightarrow[]{\Delta} B\},
	\end{equation*}where \(A,B\subset \Zd\), \(\Delta\subset E\) are finite.
\end{definition}

\begin{definition}[Insertion tolerance]
	\label{def:insertion_tolerance}
	A percolation measure \(P\) on \(E\) is said to have the \emph{insertion tolerance} property if for any edge \(e\in E\) there exists \(\theta_e>0\) such that
	\begin{equation*}
		P(\omega_e =1 \given \calF_{E\setminus\{e\}}) \geq \theta_e.
	\end{equation*}
	If \(P\) is finite range and translation invariant, it is equivalent to the existence of \(\theta>0\) such that
	\begin{equation*}
		\min_{e\in E} P(\omega_e =1 \given \calF_{E\setminus \{e\}}) \geq \theta.
	\end{equation*}
\end{definition}

A useful consequence of insertion tolerance is
\begin{lemma}
	\label{lem:insertion_tol_consequence}
	Suppose \(P\) is a finite range, translation invariant percolation measure on \(E\). Then, for any \(x,y\in \Zd\), and any sets \(A,B\subset \Zd\),
	\begin{equation*}
		P(x\leftrightarrow A, y\leftrightarrow B, x\leftrightarrow y)\geq \theta^{\rmd_{E}(x,y)} P(x\leftrightarrow A, y\leftrightarrow B),
	\end{equation*}where \(\rmd_E\) is the graph distance on \((\Zd,E)\) and \(\theta>0\) is given by Definition~\ref{def:insertion_tolerance}.
\end{lemma}
\begin{proof}
	Let \(\gamma\) be a path (seen as set of edges) from \(x\) to \(y\) realizing \(\rmd_{E}(x,y)\) (in particular, \(|\gamma| = \rmd_{E}(x,y)\)). Now,
	\begin{align*}
		P(x\leftrightarrow A, y\leftrightarrow B, x\leftrightarrow y) &\geq P\big(P(x\leftrightarrow A, y\leftrightarrow B, \gamma\subset \omega\given \calF_{E\setminus \gamma}) \big)\\
		&= P\big(\mathds{1}_{\omega\cup \gamma \in \{x\leftrightarrow A\} }\mathds{1}_{\omega\cup \gamma \in \{y\leftrightarrow B\}}  P(\gamma\subset \omega\given \calF_{E\setminus \gamma}) \big)\\
		&\geq P\big(\mathds{1}_{\omega \in \{x\leftrightarrow A\} }\mathds{1}_{\omega \in \{y\leftrightarrow B\}} \big)\theta^{|\gamma|}.
	\end{align*}
\end{proof}

We will regularly use this kind of argument without explicitly writing down the details.

\section{Coarser lattice, restricted connections, preliminary results}

In all this Section, we work under the hypotheses of Theorem~\ref{thm:existence}.

\subsection{Coarse connections}
To avoid dealing with trivialities occurring from the discrete structure of \(\Zd\), we will look at a coarser notion of connections. Let \(R_0\geq R\) (recall \(R\) is the range) be a fixed integer such that \(\big(\Lambda_{R_0}, E(\Lambda_{R_0})\big)\) is connected. Denote \(\Gamma = ((2R_0+1)\bbZ)^{d}\) the coarser lattice. For \(x\in \Gamma\), denote \(\Lambda(x)= x+\Lambda_{R_0}\). To lighten notations, for \(x,y\in\Gamma\) we will write \(\{x\leftrightarrow y\}\) for the event \(\{\Lambda(x)\leftrightarrow \Lambda(y)\}\). By Lemma~\ref{lem:insertion_tol_consequence}, these events have the same asymptotic decay rates as the point-to-point rates.

For a point \(x\in\Rd\), denote \(B_x = v_x+[-R_0-1/2,R_0+1/2)^d\) the box such that \(v_x\in\Gamma\), \(x\in B_x\). For a set \(\Delta\subset \R^d\), we denote \([\Delta] = \bigcup_{x\in \Delta} B_x \cap\Zd\). For \(x,y\in\R^d, \Delta\subset \Rd\), we write \(\{x\xleftrightarrow[]{\Delta} y\} = \{\Lambda(v_x)\xleftrightarrow[]{E([\Delta])} \Lambda(v_y)\}\).

In the same spirit, for \(\Delta\subset\Rd\), we say that an event is \(\Delta\)-measurable if it is in \(\calF_{E([\Delta])}\).

\subsection{A family of coarse graining}
\label{sec:coarse_graining}

We will regularly use coarse-graining of the cluster of \(0\). We describe here a generic coarse-graining procedure parametrized by the ``unit cell'' of the coarse graining. These procedures are a general formulation of the coarse-graining procedure applied in~\cite{Campanino+Ioffe+Velenik-2008}. Let \(0\in\Delta\subset \Zd\) be finite. Let \(\Delta_K= \bigcup_{x\in \Delta} \Lambda_K(x)\). Let \(\calT = \calT(\Delta,K)\) be the set of embedded rooted trees defined as follows: \(T\in\calT\) is the data of a set of vertices \(t=\{t_0,\cdots,t_m\}\) where each \(t_i\in\Zd\), and a set of edges \(f=\{f_1,\cdots, f_{m}\}\) with \(f_i\subset t, |f_i|=2\) such that
\begin{itemize}
	\item The graph \((t,f)\) is a tree.
	\item A given point in \(\Zd\) can only occur once as element of \(t\).
	\item \(t_0 =0\), \(t_{i}\in \partial^{\exterior} (\Delta_K+ t_j)\) where \(f_i=\{t_i,t_j\}\).
	\item The labels and edges can be inductively reconstructed from the set of vertices (without labels) \(W\) as follows: \(t_i\) is the smallest (for the fixed total order on \(\Zd\)) element of \(W\setminus \{0,t_1,\cdots,t_{i-1}\}\) belonging to \( \bigcup_{j=0}^{i-1} \partial^{\exterior}(t_{j}+\Delta_K) \) and \(f_i\) is given by \(\{t_i,v^*\}\) where \(v^*\) is the smallest element of \(\{t_0,\cdots, t_{i-1}\}\) with \(t_i\in \partial^{\exterior}(v^*+\Delta_K) \).
\end{itemize}

A fairly direct observation is that the degree of a vertex \(t_i\) in \((t,f)\) is less than \(d_{\Delta_K} = |\partial^{\exterior} \Delta_K |\) and one has a natural inclusion of \(\calT_l = \{T\in \calT:\ |t|=l \}\) in the set of sub-trees of \(\bbT_{d_{\Delta_K}}\) (the \(d_{\Delta_K}\)-regular tree) containing \(0\) and having \(l\) vertices. In particular, there exists \(c>0\) universal such that \(|\calT_l|\leq e^{c \log(d_{\Delta_K})l}\).

We now define a mapping \(\CG_{\Delta,K}\) from the set of clusters containing \(0\) to \(\calT(\Delta,K)\). We define it via an algorithm constructing \(T\in\calT\) from \(C\ni 0\) (see Figure~\ref{fig:coarse_graining_square_cell_expl}).
Fix some \(C\ni 0\). Consider the graph formed by the vertices of \(\Zd\) and the edges in \(C\). Construct \(t,f\) as follows

\begin{algorithm}[H]
	\label{alg:CGpercolation}
	Set \(t_0=0\), \(t=\{t_0\}\), \(f=\varnothing\), \(V= \Delta_K\), \(i=1\)\;
	\While{\(A=\big\{ z\in\partial^{\exterior}V:\ z\xleftrightarrow{(z+\Delta)\setminus V }\partial^{\exterior}(z+\Delta)\big\}\neq \varnothing \)}{
		Set \(t_{i}= \min A \)\;
		Let \(v_*\) be the smallest \(v \in t\) such that \(t_m\in \partial^{\exterior}(\Delta + v^*)\)\;
		Set \(f_i=\{v^*,t_i\}\)\;
		Update \(t=t\cup\{t_{i}\}\), \(f=f\cup\{f_i\}\), \(V= V\cup(t_i + \Delta_K)\), \(i=i+1\)\;
	}
	Set \(m=i\)\;
	\Return \((t,f)\)\;
	\caption{Coarse graining of a cluster containing \(0\).}
\end{algorithm} Write \(\CG_{\Delta,K}(C) =(t(C),f(C))\). One has automatically that \(C\) is in a \linebreak\( (K +2\times\mathrm{radius}(\Delta))\)-neighbourhood of \(\CG_{\Delta,K}(C)\).

\begin{figure}[h]
	\centering
	\includegraphics[scale=0.8]{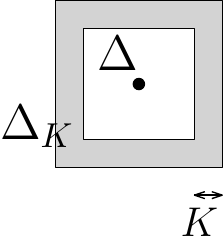}
	\hspace{20mm}
	\includegraphics[scale=0.8]{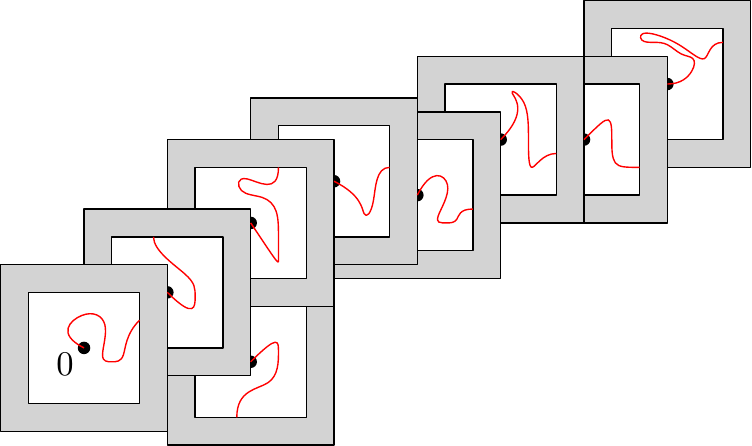}
	\caption{Left: a possible cell \(\Delta\). Right: a coarse graining using the square cell. Required connections are depicted in red.}
	\label{fig:coarse_graining_square_cell_expl}
\end{figure}

The usefulness of such coarse graining is the conjunction of the combinatorial control we mentioned on trees with given number of vertices and the following energy bound.
\begin{lemma}
	\label{lem:tree_energy_bound}
	Suppose the hypotheses of Theorem~\ref{thm:existence} hold. Then, there exists \(K_0\geq 0\) such that for any \(0\in\Delta\subset \Zd\) finite, \(K\geq K_0\),
	and \(T=(t,f)\in\calT(\Delta,K)\),
	\begin{equation*}
	P\big(\CG_{\Delta,K}(C_0) = T\big)\leq \Big(P(0\leftrightarrow \Delta^c)(1+|\Delta|e^{-\rateMix K/2})\Big)^{|f|}.
	\end{equation*}
\end{lemma}
\begin{proof}
	Let \(T=(t,f)\). The event \(\CG_{\Delta,K}(C_0) = T\) implies in particular that
	\begin{equation*}
	\bigcap_{i=1}^{|f|} \{t_i\xleftrightarrow[]{(t_i+\Delta)\setminus V_i} \partial^{\exterior}(t_i+\Delta)\} \equiv \bigcap_{i=1}^{|f|} A_i
	\end{equation*}occurs, where \(V_i = \bigcup_{0\leq j<i} (t_j+\Delta_K)\). Now, let \(F_i\) denote the support of \(A_i\). One has that \(|F_i|\leq C|\Delta|\) for any \(i\) (recall \(P\) has finite range) and \(\rmd(F_i,F_j)\geq K\). In particular, by~\eqref{eq:def:ratio_mix},
	\begin{align*}
	P(\bigcap_{i=1}^{|f|} A_i) &\leq P(\bigcap_{i=1}^{|f|-1} A_i)P(A_{|f|})\Big(1+\cstMix\sum_{e\in F_{|f|}, e': \rmd(e,e')\geq K} e^{-\rateMix\rmd(e,e')}\Big)\\
	&\leq P(\bigcap_{i=1}^{|f|-1} A_i)P(A_{|f|})\Big(1+C|\Delta| K^{d-1}e^{-\rateMix K}\Big)\\
	&\leq P(\bigcap_{i=1}^{|f|-1} A_i)P(0\leftrightarrow \Delta^c) \Big(1+C|\Delta| K^{d-1}e^{-\rateMix K}\Big)
	\end{align*}where we used inclusion of events and translation invariance in the last line. Iterating \(|f|\) times gives the result.
\end{proof}

\section{Proofs}

The proof will go by introducing a family of decay rates (rates associated to various connection events). The idea is to prove the wanted properties for convenient rates and then to prove that all rates are in fact the same. Again, we work under the hypotheses of Theorem~\ref{thm:existence} which are implicitly assumed in the statements.

\subsection{Constraint point-to-point}

First introduce a family of connection events. For \(\delta\in(0,1]\) and \(s,s'\in\bbS^{d-1}\) such that \(s\in\mathring{\cone}_{s',\delta}\),
\begin{equation*}
	Q_{s',\delta}(s,N) = \{ 0\xleftrightarrow[]{\cone_{s',\delta}\setminus H_{s'}(Ns) } Ns\}.
\end{equation*}

\begin{lemma}
	\label{lem:constraint_pt2pt_exist}
	For any \(\delta\in(0,1]\) and \(s,s'\in\bbS^{d-1}\) such that \(s\in\mathring{\cone}_{s',\delta}\), the limit
	\begin{equation*}
		\tilde{\nu}_{s',\delta}(s) = \lim_{N\to\infty} -\frac{1}{N} \log P(Q_{s',\delta}(s,N))
	\end{equation*}exists.
\end{lemma}
\begin{proof}
	Fix \(s,s'\in\bbS^{d-1}\), \(\delta\in(0,1]\) such that \(s\in\mathring{\cone}_{s',\delta}\). By assumption, \(P(0\leftrightarrow \Lambda_n^c)\leq e^{-\rateCo n}\). Denote \(l=2\limsup_{N\to\infty}-\frac{1}{N} \log P(Q_{s',\delta}(s,N)) \) and set \(\alpha= \frac{l}{\rateCo}\). In particular, there exists \(N_0\) such that for any \(N\geq N_0\), \( P(Q_{s',\delta}(s,N) ) \geq e^{-lN}\). So, \(-\frac{1}{N} \log P(Q_{s',\delta}(s,N))\) has the same upper and lower limits as the sequence
	\begin{equation*}
		\frac{a_N}{N} = -\frac{1}{N}\log P\big(0\xleftrightarrow[]{(\cone_{s',\delta}\setminus H_{s'}(Ns))\cap\Lambda_{\alpha N} } Ns\big).
	\end{equation*} See Figure~\ref{fig:pt_to_pt_constraint_exists_intersect_Lambda} for the volume in which the connection is required to occur.
	\begin{figure}[h]
		\centering
		\includegraphics[scale=0.8]{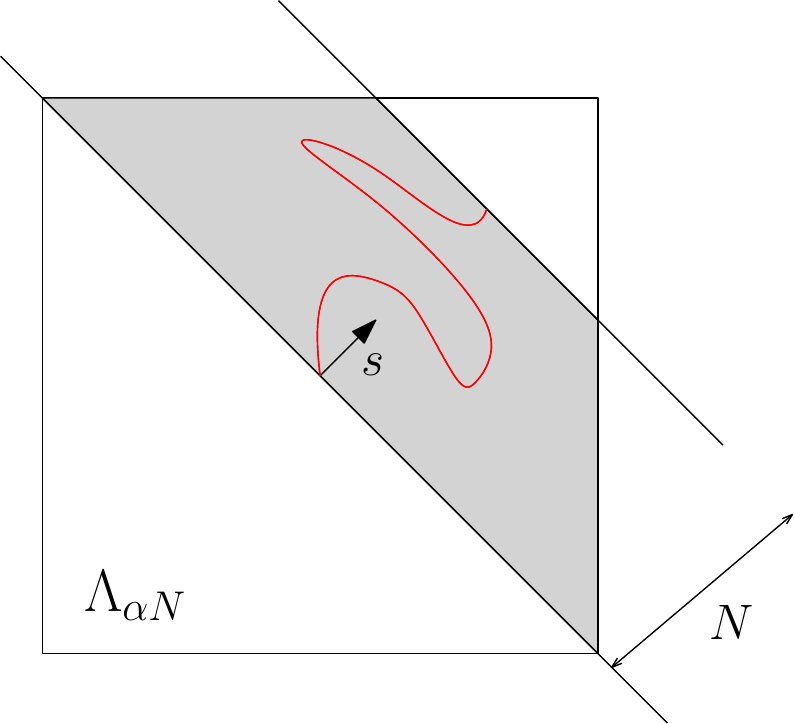}
		\caption{The volume \((\cone_{s,1}\setminus H_{s}(Ns))\cap\Lambda_{\alpha N}\).}
		\label{fig:pt_to_pt_constraint_exists_intersect_Lambda}
	\end{figure}
	This additional manipulation is only needed to handle \(\delta=1\), see Figure~\ref{fig:pt_to_pt_constraint_exists}. We show that \(a_N\) satisfies the hypotheses of Lemma~\ref{app:quasi_sub_add}. Let \(\Delta_N= (\cone_{s',\delta}\setminus H_{s'}(Ns))\cap\Lambda_{\alpha N}\). Let \(n\geq m\) be large enough, \(\ell= \log(m)^2\), and set \(N= n+m+\ell\). Then, \(\Delta_n\subset\Delta_{N}\), \(\big((n+\ell)s+\Delta_m\big)\subset\Delta_{N}\), and \(\rmd(\Delta_n,(n+\ell)s+\Delta_m)\geq \ell\). Then,
	\begin{equation*}
		P\big(0\xleftrightarrow[]{\Delta_{N} } Ns\big)\geq \theta^{\cstEquivDist \ell }P\big(0\xleftrightarrow[]{\Delta_{n} } ns, (n+\ell)s\xleftrightarrow[]{(n+\ell)s+\Delta_m } Ns\big).
	\end{equation*}by inclusion of events and Lemma~\ref{lem:insertion_tol_consequence}. Then, ratio mixing implies
	\begin{equation*}
		P\big(0\xleftrightarrow[]{\Delta_{n} } ns, (n+\ell)s\xleftrightarrow[]{(n+\ell)s+\Delta_m } Ns\big) \geq (1-|\Delta_m|e^{-\rateMix\ell/2}) P\big(0\xleftrightarrow[]{\Delta_{n} } ns\big)P\big(0\xleftrightarrow[]{\Delta_m } ms\big)
	\end{equation*} for any \(m\) large enough. \(|\Delta_m|\) being upper bounded by a degree \(d\) polynomial in \(m\), the wanted property follows with \(g(m) = \log(m)^2\), and \(f(m) = 2 + \cstEquivDist\log(\theta^{-1}) \log(m)^2 \).
	\begin{figure}[h]
		\centering
		\includegraphics[scale=0.65]{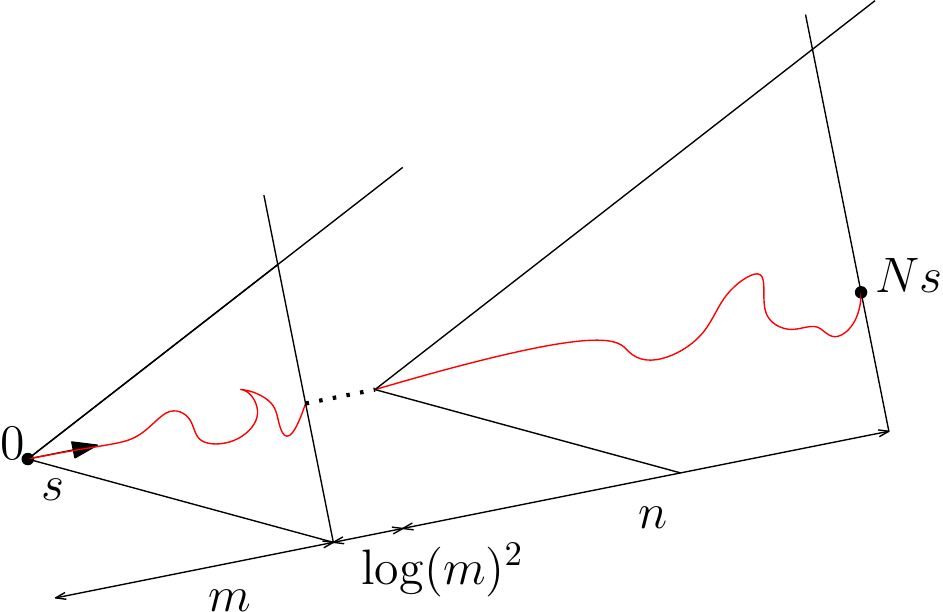}
		\hspace{5mm}
		\includegraphics[scale=0.65]{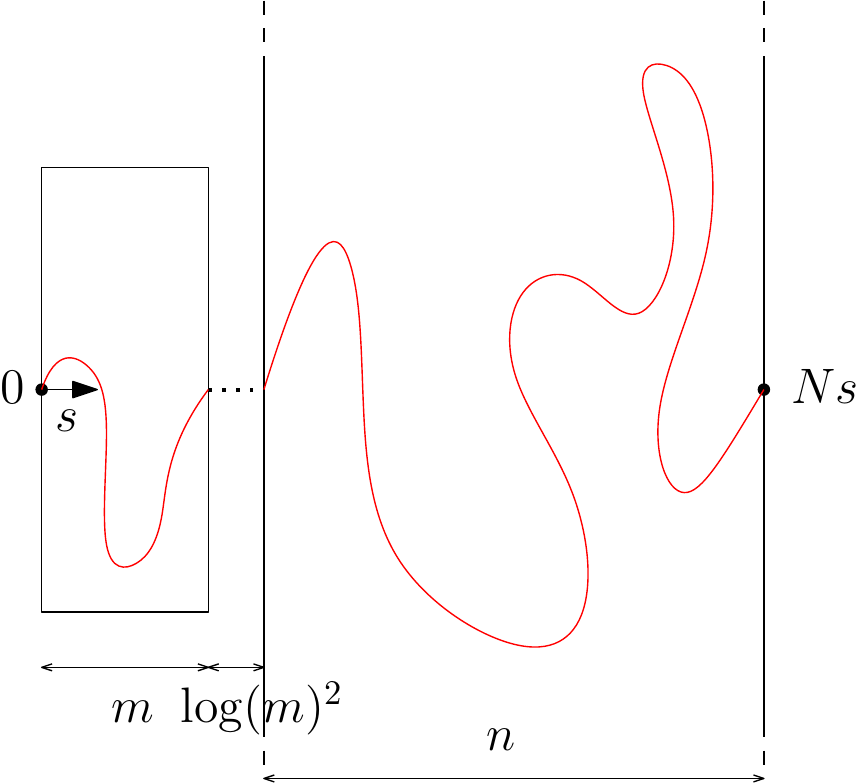}
		\caption{Left: construction of the local event when \(\delta<1\). Right: construction of the local event when \(\delta=1\). Dotted lines denote the use of insertion tolerance.}
		\label{fig:pt_to_pt_constraint_exists}
	\end{figure}
\end{proof}

\begin{lemma}
	\label{lem:constraint_pt2pt_s_delta_indep}
	For any \(s\in\bbS^{d-1}\), \(\tilde{\nu}_{s',\delta}(s)\) does not depend on \(\delta\in(0,1]\) and \(s'\in\bbS^{d-1}\) as long as \(s\in\mathring{\cone}_{s',\delta}\).
\end{lemma}
\begin{proof}
	Fix \(s\in\bbS^{d-1}\) and omit it from notation. Let \(\delta',\delta''\in(0,1]\) and \(s',s''\in \bbS^{d-1}\) be such that \(s\in \mathring{\cone}_{s',\delta'} \cap \mathring{\cone}_{s'',\delta''}\). To lighten notation, write \(r'=\tilde{\nu}_{s',\delta'}\) and \(r''=\tilde{\nu}_{s'',\delta''}\). We first prove \(r'\leq r''\). Let \(\alpha = \frac{2r'' }{\rateCo}\). In particular, defining \(\Delta_n = (\cone_{s'',\delta''} \setminus H_{s''}(ns)) \cap \Lambda_{\alpha n}\) (see Figure~\ref{fig:pt_to_pt_constraint_equiv_construct_cell}),
	\begin{equation*}
		P(0\xleftrightarrow[]{\Delta_n} ns) = e^{-r'' n(1+o_n(1))}.
	\end{equation*}
	\begin{figure}[h]
		\centering
		\includegraphics[scale=0.55]{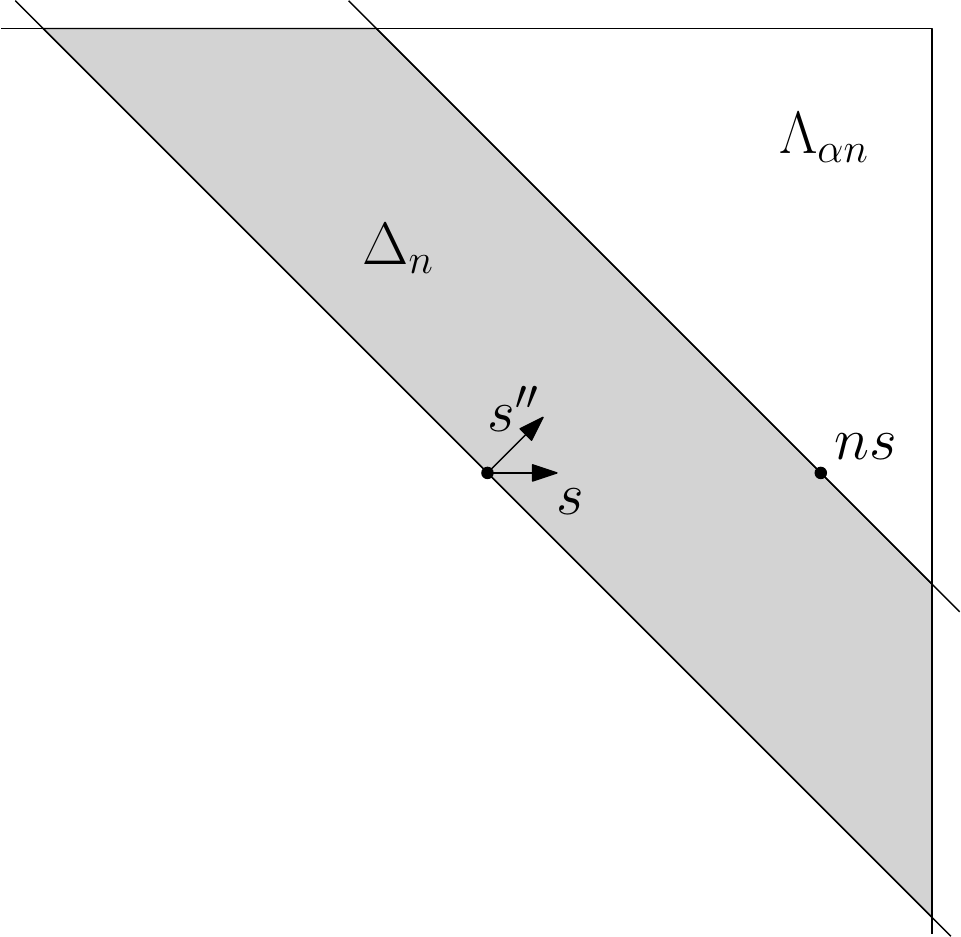}
		\caption{The volume \(\Delta_n\) when \(\delta'' =1\).}
		\label{fig:pt_to_pt_constraint_equiv_construct_cell}
	\end{figure}
	Then, fix \(\epsilon>0\) small and \(n\) large enough. Write \(\ell=\log(n)^2\). For any \(N\) large, \((1-\epsilon)N=q(n+\ell) + b\) with \(b<n+\ell\) (integer parts are implicitly taken). One has
	\begin{multline*}
		P\big(Q_{s',\delta'}(s,N)\big) \geq\\\geq \theta^{\cstEquivDist \epsilon N + b + q\ell} P\Big( \bigcap_{i=0}^{q-1} \big\{ (\frac{\epsilon}{2} N + i(n+\ell)) s \xleftrightarrow[]{(\epsilon N + i(n+\ell)) s + \Delta_n} (\frac{\epsilon}{2} N + i(n+\ell)+ n) s \big\}  \Big),
	\end{multline*}where we used insertion tolerance (Lemma~\ref{lem:insertion_tol_consequence}). See Figure~\ref{fig:pt_to_pt_constraint_equiv_construct_LB}.
	\begin{figure}[h]
		\centering
		\includegraphics[scale=0.8]{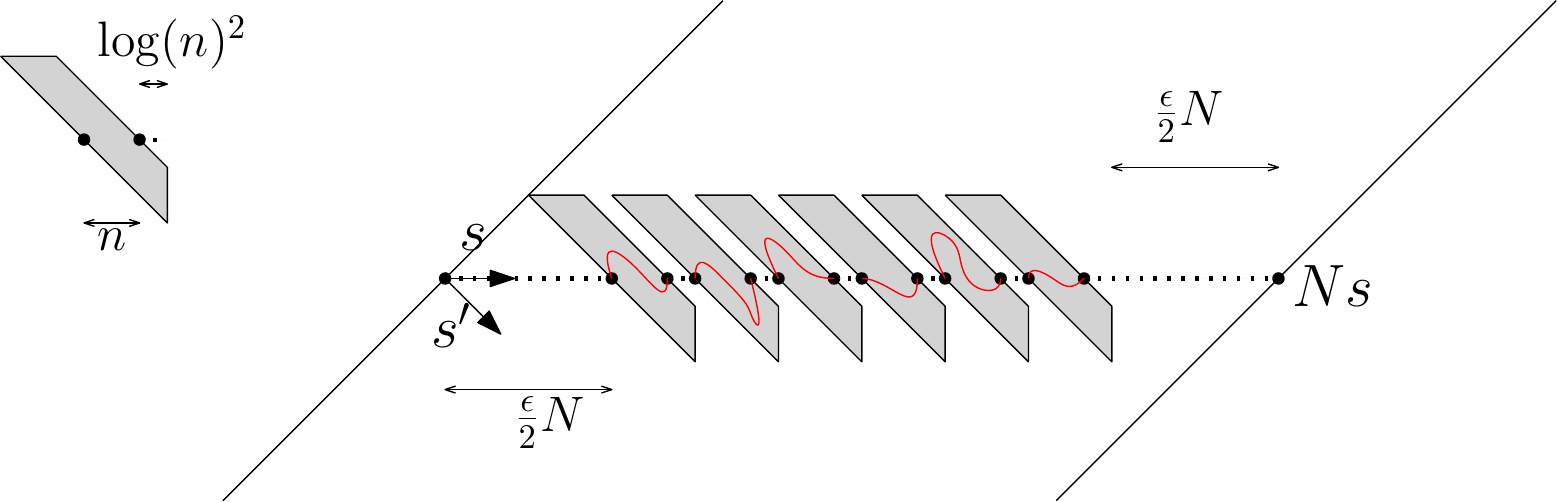}
		\caption{The construction of the lower bound. Dotted lines denote the use of insertion tolerance.}
		\label{fig:pt_to_pt_constraint_equiv_construct_LB}
	\end{figure}
	Using ratio mixing and translation invariance, the probability in the RHS is lower bounded by
	\begin{equation*}
		e^{-q}\prod_{i=0}^{q-1}P\big( 0 \xleftrightarrow[]{\Delta_n}  n s \big) = e^{-q} e^{-r''qn(1+o_n(1))}
	\end{equation*}whenever \(n\) is larger than some fixed constant. Taking the log, dividing by \(-N\) and taking \(N\to\infty\), one obtains
	\begin{equation*}
		r'\leq \cstEquivDist \log(\theta^{-1}) \epsilon + \frac{(\log(\theta^{-1}) \ell +1) (1-\epsilon)}{n+\ell} + \frac{(1-\epsilon)nr''}{n+\ell}
	\end{equation*}\(\epsilon>0\) is arbitrary and \(n\) is arbitrarily large. Take \(n\to\infty\) and then \(\epsilon\searrow 0\) to obtain the wanted inequality.
	
	Repeating the argument with \((s',\delta')\) and \((s'',\delta'')\) exchanged yields the reverse inequality and thus the result.
\end{proof}

From Lemma~\ref{lem:constraint_pt2pt_exist} and~\ref{lem:constraint_pt2pt_s_delta_indep}, it is natural to introduce \(\tilde{\nu}:\Rd\to \R_+\) as the extension by positive homogeneity of \(\tilde{\nu}_{s',\delta}(s)\).

\begin{lemma}
	\label{lem:constraint_rate_norm}
	\(\tilde{\nu}\) defines a norm on \(\Rd\).
\end{lemma}
\begin{proof}
	First, point separation follows from the exponential decay assumption (\(\rateCo>0\)). Then, positive homogeneity of order one is a direct consequence of the way we extended \(\tilde{\nu}\) to \(\Rd\) and of
	\begin{equation*}
		P\big(Q_{s,1}(s,N)\big) = P\big(Q_{-s,1}(-s,N)\big),
	\end{equation*} by translation invariance. Remains the triangle inequality. Fix \(x,y\in\Rd\). Let \(s_{xy}=\frac{x+y}{\norm{x+y}}\), \(s_x=\frac{x}{\norm{x}}\), \(s_y=\frac{y}{\norm{y}}\). We can suppose that \(x,x+y\in \mathring{H}_{s_{xy}}\) (otherwise, exchange the role of \(0\) and \(x+y\), see Figure~\ref{fig:norm_equiv_triangles}).
	\begin{figure}[h]
		\centering
		\includegraphics[scale=0.7]{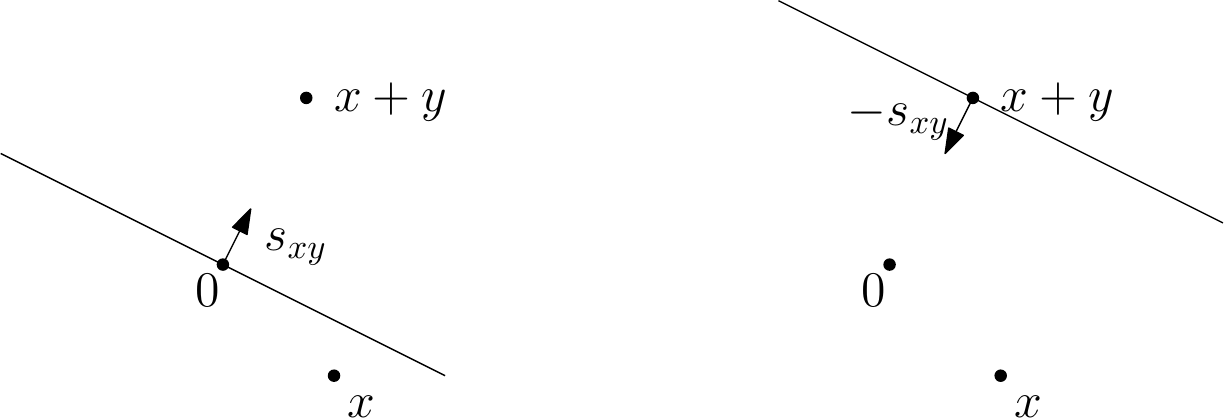}
		\caption{}
		\label{fig:norm_equiv_triangles}
	\end{figure}
	Then, for \(\epsilon>0\) fixed, for any \(\delta>0\) small enough and any \(N\) large
	\begin{multline*}
		P\big(Q_{s_{xy},1}(s_{xy},N\norm{x+y}) \big) \geq \\
		\geq \theta^{\epsilon c(\norm{x}+\norm{y}) N} P\big(Q_{s_x,\delta}(s_x,\norm{x}(1-\epsilon) N), x+y\xleftrightarrow[]{\cone_{-s_y,\delta}^{\norm{y}(1-\epsilon) N}(x+y)} x+\epsilon N\norm{y} s_y\big),
	\end{multline*}where we used insertion tolerance. See Figure~\ref{fig:norm_triangle_ineq}.
	\begin{figure}[h]
		\centering
		\includegraphics[scale=0.8]{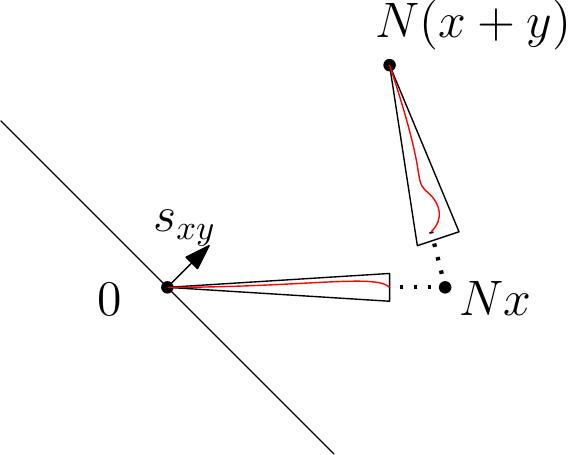}
		\caption{Construction of the forced connection through \(Nx\). Dotted lines denote the use of insertion tolerance.}
		\label{fig:norm_triangle_ineq}
	\end{figure}
	Now, for \(\epsilon>0\) fixed and \(\delta>0\) small enough (depending on \(\epsilon\)), one can use ratio mixing to obtain that the last probability is lower bounded by
	\begin{equation*}
		(1-e^{-c'\epsilon N})P\big(Q_{s_x,\delta}(s_x,\norm{x}(1-\epsilon) N)\big)P\big( Q_{s_y,\delta}(s_y,\norm{y}(1-\epsilon) N)\big).
	\end{equation*}Taking the log, dividing by \(-N\) and sending \(N\to\infty\), one obtains
	\begin{equation*}
		\norm{x+y}\tilde{\nu}(s_{xy}) \leq \log(\theta^{-1})\epsilon c(\norm{x}+\norm{y}) + (1-\epsilon)\norm{x}\tilde{\nu}(s_x) +(1-\epsilon)\norm{y}\tilde{\nu}(s_y).
	\end{equation*}\(\epsilon>0\) was arbitrary, taking \(\epsilon\searrow 0\) and using positive homogeneity gives \(\tilde{\nu}(x+y)\leq \tilde{\nu}(x)+\tilde{\nu}(y)\).
\end{proof}

\subsection{Point-to-half-space}

\begin{lemma}
	\label{lem:pt2HS_exists}
	Let \(s\in\bbS^{d-1}\). The limit
	\begin{equation*}
		\nu_H(s) = \lim_{N\to\infty}-\frac{1}{N} \log P\big(0\leftrightarrow H_s(Ns)\big)
	\end{equation*}exists.
\end{lemma}
\begin{proof}
	We fix \(s\in\bbS^{d-1}\) and omit it from the notation. Let \((n_k)_{k\geq 1}\) be an increasing sequence of integers such that
	\begin{equation*}
		\lim_{k\to\infty} -\frac{1}{n_k} \log P\big(0\leftrightarrow  H_s(n_ks)\big) =\limsup_{N\to\infty} -\frac{1}{N} \log P\big(0\leftrightarrow  H_s(Ns)\big) \equiv \overline{\nu}_H
	\end{equation*} In particular, \(P\big(0\leftrightarrow H_s(n_ks)\big) = e^{-n_k\overline{\nu}_H (1+o_k(1))}\).
	
	By our hypotheses,
	\begin{equation*}
	P(0\leftrightarrow \Lambda_M )\leq e^{-\rateCo M}
	\end{equation*}for any \(M\) large enough. Let then \(\alpha = \frac{\overline{\nu}_H}{\rateCo}\). Set \(\Delta_k= \Lambda_{\alpha n_k}\setminus H_{s}(n_k s) \), \(K_k=  \log(n_k)^2\), \(\overline{\Delta}_k = \bigcup_{v\in\Delta_k} \Lambda_{K_k}(v)\). See Figure~\ref{fig:coarse_graining_HS_Box_cell}. In particular, we have
	\begin{equation}
	\label{eq:energy_inf_vol}
	P(0\leftrightarrow \Delta_k^c) \leq P(0\leftrightarrow \Lambda_{\alpha n_k}^c ) + P\big(0\leftrightarrow H_s (n_k s) \big) \leq e^{ -\overline{\nu}_H n_k(1+o_k(1))}
	\end{equation}where we used a union bound.
	
	\begin{figure}[h]
		\centering
		\includegraphics[scale=0.8]{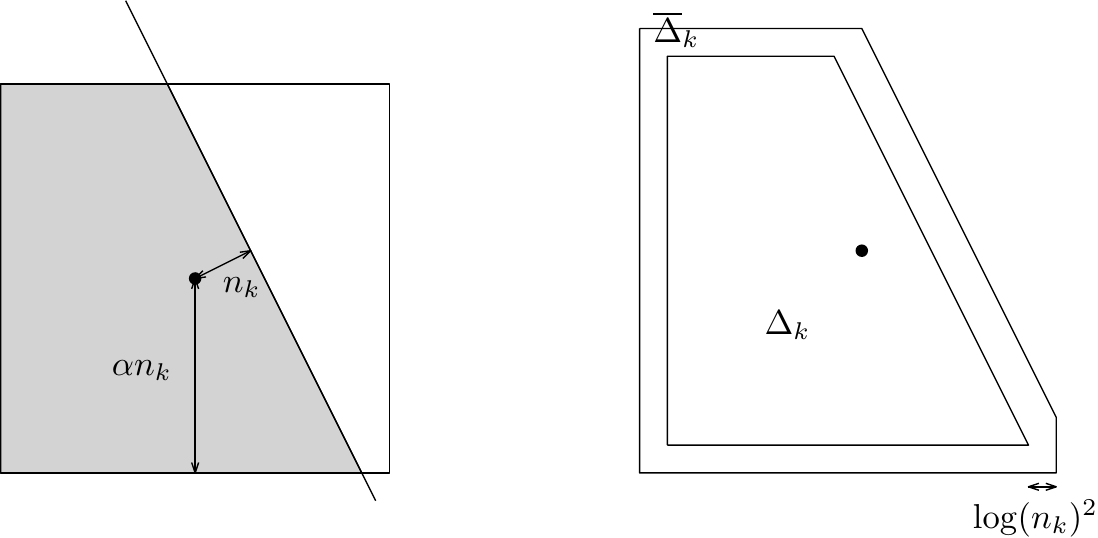}
		\caption{The cell \(\Delta_k\).}
		\label{fig:coarse_graining_HS_Box_cell}
	\end{figure}
	
	We now coarse-grain \(C_0\) using \(\CG_k\equiv \CG_{\Delta_k,K_k}\) (see Section~\ref{sec:coarse_graining}). Write \(\CG_k(C_0) =(t(C_0), f(C_0))\). One has that \(C_0\) is included in an \(3\alpha n_k\)-neighbourhood of \(t(C_0)\). We have
	\begin{equation}
	P(0\leftrightarrow X) = \sum_{T\in \calT} P\big(0\leftrightarrow X, \CG_k(C_0) = T \big) \leq \sum_{T\sim X} P\big(\CG_k(C_0) = T \big)
	\end{equation}where \(T\sim X\) means that \(\rmd(X,t(C_0))\leq 3\alpha n_k\). We can then use Lemma~\ref{lem:tree_energy_bound} and the bound on the number of trees to obtain that for any fixed large enough \(k\), as \(N\) goes to infinity,
	\begin{align*}
	P\big(0\leftrightarrow H_s(Ns)\big) &\leq \sum_{l\geq \frac{N}{n_k+K_k}}\sum_{T\in \calT_{l}} P\big(\CG_k(C_0) = T \big) \\
	&\leq  \sum_{l\geq \frac{N}{n_k +K_k}}e^{c\log(d_{\overline{\Delta}_k}) l} e^{-\overline{\nu}_H n_k l (1+o_k(1))}\\
	&=\sum_{l\geq \frac{N}{n_k+K_k}} e^{-\overline{\nu}_H n_k l (1+o_k(1) +o_{n_k}(1))} \\
	&= e^{-N\overline{\nu}_H (1+o_k(1) +o_{n_k}(1))} \big(1-o_k(1)\big)^{-1}
	\end{align*}as \(d_{\overline{\Delta}_k}\) is upper bounded by a polynomial of degree \(d\) in \(n_k\) and any tree \(T\) with \(T\sim H_s(N_s)\) has \(|f|\geq \frac{N}{n_k+K_k}\). In particular, for any \(k\) large enough,
	\begin{equation*}
	\underline{\nu}_H\equiv\liminf_{N\to\infty}-\frac{1}{N}\log P(0\leftrightarrow Ns+H_s) \geq \overline{\nu}_H (1+o_k(1) +o_{n_k}(1)).
	\end{equation*} Taking \(k\to \infty\) yields \(\underline{\nu}_H\geq \overline{\nu}_H\). The direction \(s\) being arbitrary, \(\overline{\nu}_H(s) = \underline{\nu}_H(s)\) for all \(s\in\bbS^{d-1}\). 
\end{proof}

\subsection{Constrained point-to-half-space}

\begin{lemma}
	\label{lem:constraint_pt2HS_exists}
	Let \(s\in\bbS^{d-1}\). The limit
	\begin{equation*}
	\tilde{\nu}_H(s) = \lim_{N\to\infty}-\frac{1}{N} \log P\big(0\xleftrightarrow[]{H_s} H_s(Ns)\big)
	\end{equation*}exists. Moreover,
	\begin{equation*}
		\tilde{\nu}_H(s) = \nu_H(s).
	\end{equation*}
\end{lemma}
\begin{proof}
	We fix \(s\in\bbS^{d-1}\) and omit it from the notation. By inclusion of events, one has \(\liminf_{N\to\infty}-\frac{1}{N} \log P\big(0\xleftrightarrow[]{H_s} H_s(Ns)\big) \geq \nu_H\). To obtain the other bound, start with, for any \(\epsilon>0\),
	\begin{equation}
	\label{eq:pt2HS_eq1}
		P\big(0\leftrightarrow H_s(Ns)\big) \leq \theta^{-\cstEquivDist \epsilon N}P(\epsilon Ns \leftrightarrow H_s(Ns)) = e^{ \lambda \epsilon N}e^{-\nu_H(1-\epsilon)N(1+o_N(1))},
	\end{equation}where \(\lambda = \log(\theta^{-1})\cstEquivDist>0\) and
	\begin{equation}
	\label{eq:pt2HS_eq2}
		P\big(0\xleftrightarrow[]{H_s} H_s(Ns)\big) \geq e^{-\lambda\epsilon N}P(\epsilon Ns \xleftrightarrow[]{H_s} H_s(Ns)),
	\end{equation}by Lemma~\ref{lem:insertion_tol_consequence} (insertion tolerance).
	
	We then use a coarse-graining described in Section~\ref{sec:coarse_graining} (the same as in the proof of Lemma~\ref{lem:pt2HS_exists} with different sizes). Set \(\Delta_n = \Lambda_{\alpha n}\setminus H_{s}(n s)\), \(K_n= \log(n)^2\), and \(\overline{\Delta}_n= \bigcup_{v\in\Delta_n}\Lambda_{K_n}(v)\), where \(\alpha\) is the same quantity as in the proof of Lemma~\ref{lem:pt2HS_exists}. As in Lemma~\ref{lem:pt2HS_exists},
	\begin{equation*}
		P(0\leftrightarrow \Delta_n^c) \leq e^{-\nu_Hn(1+o_n(1))}.
	\end{equation*}We use \(\CG_n\equiv \CG_{\Delta_n,K_n}\). Write \(\CG_n(C_0) = \big(t(C_0),f(C_0)\big)\).
	
	Now, any cluster contributing to \(\{\epsilon Ns \leftrightarrow H_s(Ns)\}\setminus \{\epsilon Ns \xleftrightarrow[]{H_s} H_s(Ns)\}\) has \(|f|\geq \frac{\epsilon N}{\sqrt{d}(\alpha n+\log(n)^2)} + \frac{(1-\epsilon) N}{n+\log(n)^2} \) (see Figure~\ref{fig:coarse_graining_HS_Box_CG_no_backtrack}). So, applying the same argument as in Lemma~\ref{lem:constraint_pt2HS_exists},
	\begin{equation*}
		P(\epsilon Ns \leftrightarrow H_s(Ns)) - P(\epsilon Ns \xleftrightarrow[]{H_s} H_s(Ns)) \leq e^{-\nu_H N (1-\epsilon + \frac{\epsilon}{\sqrt{d}\alpha}+o_n(1))}.
	\end{equation*}In particular, for any fixed \(n\) large enough, and any \(N\) large
	\begin{equation*}
		\frac{P(\epsilon Ns \xleftrightarrow[]{H_s} H_s(Ns))}{P(\epsilon Ns \leftrightarrow H_s(Ns))} \geq 1-e^{-\nu_H N (\epsilon' +o_n(1)+o_N(1))},
	\end{equation*} where \(\epsilon'=\frac{\epsilon}{\sqrt{d}\alpha}\). Plugging this in~\eqref{eq:pt2HS_eq2}, and using~\eqref{eq:pt2HS_eq1}, one obtains
	\begin{align*}
		P\big(0\xleftrightarrow[]{H_s} H_s(Ns)\big) &\geq e^{-2\lambda\epsilon N}(1-e^{-\nu_H N (\epsilon' +o_n(1)+o_N(1))})P(0 \leftrightarrow H_s(Ns))\\
		& =  e^{-2\lambda\epsilon N}(1-e^{-\nu_H N (\epsilon' +o_n(1)+o_N(1))})e^{-\nu_HN(1+o_N(1))}.
	\end{align*}
	In particular \(\limsup_{N\to\infty} -\frac{1}{N}\log P(0\xleftrightarrow[]{H_s} H_s(Ns)) \leq \nu_H +2\lambda\epsilon\). \(\epsilon>0\) being arbitrary, taking \(\epsilon\searrow 0\) yields the result.
	\begin{figure}[h]
		\centering
		\includegraphics[scale=0.8]{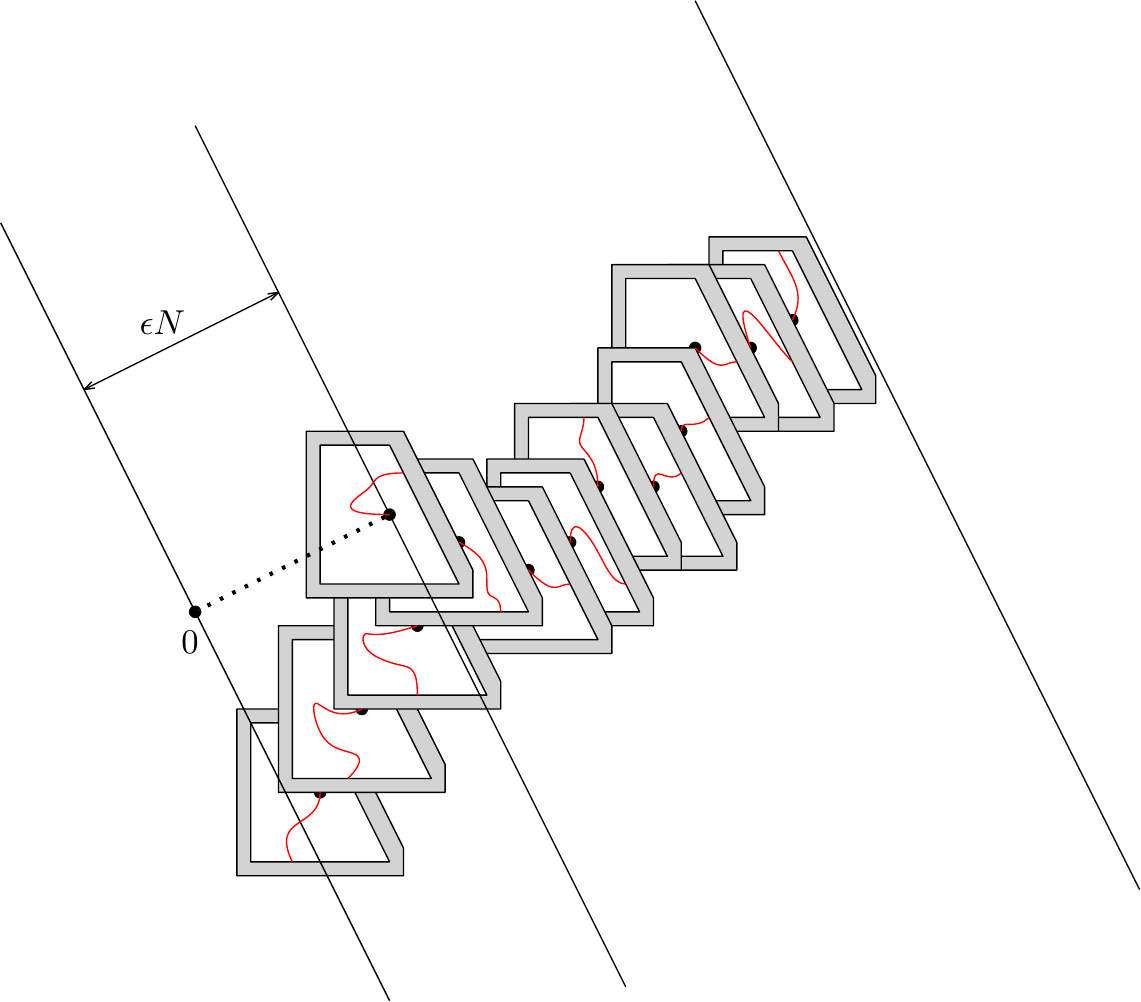}
		\caption{Coarse graining of a cluster contributing to \(\{\epsilon Ns \leftrightarrow H_s(Ns)\}\setminus \{\epsilon Ns \xleftrightarrow[]{H_s} H_s(Ns)\}\).}
		\label{fig:coarse_graining_HS_Box_CG_no_backtrack}
	\end{figure}
\end{proof}

We highlight at this point that we could easily remove the ``directed constraint'' for point-to-half-spaces connections, which seems to be much harder to do for point-to-point connections.

\subsection{Convex duality}

We saw that \(\tilde{\nu}\) defines a norm on \(\Rd\). In particular, \(\calU_{\tilde{\nu}}\) (the unit ball for \(\tilde{\nu}\)) is a convex set. To each \(s\in\bbS^{d-1}\), we associate the set of \emph{dual directions}
\begin{equation*}
s^{\star} = \Big\{s'\in\bbS^{d-1}:\ H_{s'}\big(\frac{\lrangle{s,s'}}{\tilde{\nu}(s)} s'\big)\cap \calU_{\tilde{\nu}} \subset\partial\calU_{\tilde{\nu}} \Big\}.
\end{equation*}It is the set of directions normal to the boundary of half-spaces tangent to \(\calU_{\tilde{\nu}}\) at \(\frac{s}{\tilde{\nu}(s)}\) (see Figure~\ref{fig:Convexity_duality}). By abuse of notation, we will write \(s^{\star}\) for an arbitrarily chosen element of the set. It satisfies \(\lrangle{s,s^{\star}}>0\). Moreover, for a fixed \(s^{\star}\), any \(s\) having \(s^{\star}\) as dual is a minimizer of \(s'\mapsto \frac{\tilde{\nu}(s') }{\lrangle{s^{\star},s'}}\) under the constraint \(\lrangle{s^{\star},s'}>0\). Notice that this notion of duality is not the classical convex duality between \(\calU_{\tilde{\nu}}\) and \(\calW_{\tilde{\nu}}\) (but it is related via normalization of the dual directions).

\begin{figure}[h]
	\centering
	\includegraphics[scale=0.8]{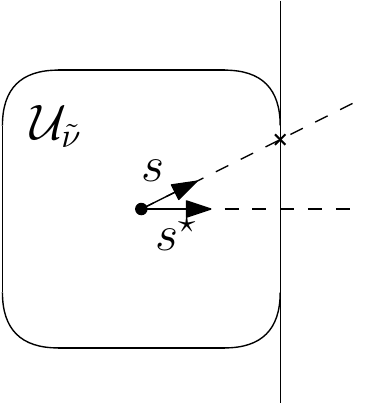}
	\caption{Duality between directions.}
	\label{fig:Convexity_duality}
\end{figure}

The duality statement is
\begin{lemma}
	\label{lem:duality}
	For any \(s\in\bbS^{d-1}\),
	\begin{equation}
	\tilde{\nu}(s) = \nu_H(s^{\star})\lrangle{s,s^{\star}}.
	\end{equation}
\end{lemma}
\begin{proof}
	Fix \(s\in\bbS^{d-1}\). Let \(s^{\star}\) be a dual direction of \(s\). Start with the easy inequality. By inclusion of events and Lemma~\ref{lem:constraint_pt2pt_s_delta_indep},
	\begin{equation*}
	P\big(0\leftrightarrow H_{s^{\star}}(Ns^{\star})\big) \geq P\big(0\xleftrightarrow[]{H_{s^{\star}}\setminus H_{s^{\star}}( a N s) } a N s\big) = e^{-a\tilde{\nu}(s)N(1+o_N(1)) }
	\end{equation*}where \(a= \lrangle{s,s^{\star}}^{-1}\). Taking the log, dividing by \(-N\) and letting \(N\to\infty\), one gets \(\nu_H(s^{\star}) \leq a\tilde{\nu}(s)\).
	
	We now proceed to the harder inequality. We use Lemma~\ref{lem:constraint_pt2HS_exists}. The idea is illustrated in Figure~\ref{fig:Convexity_pt_to_HS_favourite_dir}. Then, using the same argument as in the proof of Lemma~\ref{lem:pt2HS_exists}, for some \(\alpha\) large enough,
	\begin{equation*}
		P\big(0\xleftrightarrow[]{H_{s^{\star}}} H_{s^{\star}}(Ns^{\star})\big) \leq C P\big(0\xleftrightarrow[]{H_{s^{\star}}\cap\Lambda_{\alpha N}} H_{s^{\star}}(Ns^{\star}) \big).
	\end{equation*}By a union bound, this is in turn upper bounded by
	\begin{equation}
	\label{eq:constraint_PtoP_to_PtoHS:1}
	C\sum_{x\in \partial^{\interior}[H_{s^{\star}}(Ns^{\star})]\cap \Lambda_{\alpha N} } P(0\xleftrightarrow[]{H_{s^{\star}}\setminus H_{s^{\star}}(x) } x).
	\end{equation}

	Let \(\delta<1\) be such that \(\partial^{\interior}[H_{s^{\star}}(Ns^{\star})]\cap \Lambda_{\alpha N} \subset \cone_{s^{\star},\delta} \) for any \(N\) large enough. Let \(\epsilon>0\) be small. Choose a finite subset \(S\) of \(\bbS^{d-1}\cap \cone_{s^{\star},\delta}\) such that \(|S|\leq c''\epsilon^{-d+1}\) and \(\cone_{s^{\star},\delta} \subset\bigcup_{s'\in S} \cone_{s',\epsilon} \). Denote \(A_{s'}(N) = \partial^{\interior}[Ns^{\star}+H_{s^{\star}}]\cap \cone_{s',\epsilon}\). Then, by insertion tolerance,~\eqref{eq:constraint_PtoP_to_PtoHS:1} is upper bounded by
	\begin{equation*}
	C\sum_{s'\in S}\sum_{x \in A_{s'}(N) } \theta^{-c'\epsilon N} P(0\xleftrightarrow[]{H_{s^{\star}}\setminus H_{s^{\star}}(a_{s'}N s')} a_{s'}N s' )
	\end{equation*}with \(a_{s'}=\lrangle{s',s^{\star}}^{-1}\). By Lemma~\ref{lem:constraint_pt2pt_s_delta_indep}, \(P(0\xleftrightarrow[]{H_{s^{\star}}\setminus H_{s^{\star}}(a_{s'}N s')} a_{s'}N s' ) = e^{-a_{s'}N\tilde{\nu}(s')(1+o_N(1)) }\) with the \(o_N(1)\) depending on \(s'\). Denote it \(o^{s'}_N(1)\). Now, \(a_{s'}\tilde{\nu}(s')\) is minimal if \(s',s^{\star}\) are dual directions. So, combining all the previous observations,
	\begin{equation*}
		P\big(0\xleftrightarrow[]{H_{s^{\star}}} H_{s^{\star}}(Ns^{\star})\big) \leq C' N^{d-1} \epsilon^{1-d} \theta^{-c'\epsilon N} e^{a_{s}N\tilde{\nu}(s) \max_{s'\in S}o_N^{s'}(1)}e^{-a_{s}N\tilde{\nu}(s)}.
	\end{equation*} Taking the log, dividing by \(-N\) and taking \(N\to\infty\) gives
	\begin{equation*}
		\nu_H(s^{\star}) \geq \log(\theta)c' \epsilon  + a_{s}\tilde{\nu}(s).
	\end{equation*} Taking then \(\epsilon\searrow 0\) yields the result.
	
	\begin{figure}[h]
		\centering
		\includegraphics[scale=0.75]{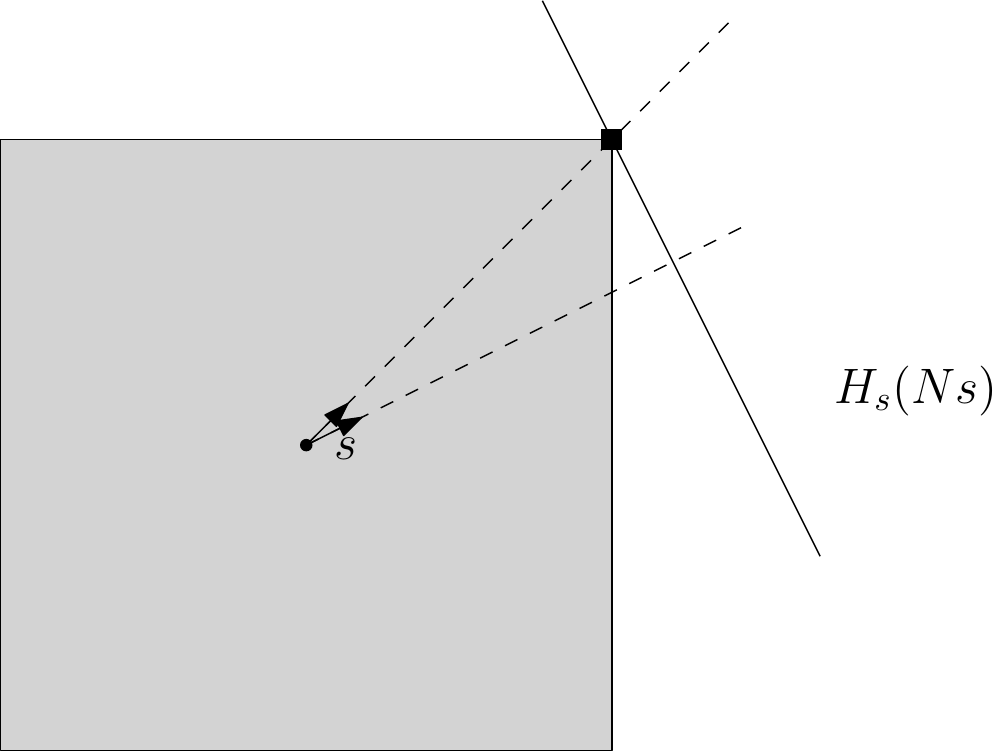}
		\caption{Connection to \(H_s(Ns)\) is made at the point minimizing the distance measured with \(\tilde{\nu}\) (here the square mark). The grey square is a dilation of \(\calU_{\tilde{\nu}}\).}
		\label{fig:Convexity_pt_to_HS_favourite_dir}
	\end{figure}
\end{proof}

\subsection{Final coarse-graining}

Let us summarize what we did so far. First, we constructed a norm using a directed version of the point-to-point connections (Lemmas~\ref{lem:constraint_pt2pt_exist},~\ref{lem:constraint_pt2pt_s_delta_indep}, and~\ref{lem:constraint_rate_norm}). Then, we proved an equivalence (at the level of exponential rates) between directed and un-directed point-to-half-space connections (Lemmas~\ref{lem:pt2HS_exists} and~\ref{lem:constraint_pt2HS_exists}). Finally, we related these two quantities using convex duality (Lemma~\ref{lem:duality}). We can now gather these three results to prove our key estimate
\begin{lemma}
	\label{lem:main_energy}
	For any \(\epsilon>0\), there exists \(L_0\geq 0\) such that for any \(L\geq L_0\),
	\begin{equation}
	P\big(0\leftrightarrow (L\calU_{\tilde{\nu}})^c \big)\leq e^{-L(1-\epsilon)}.
	\end{equation}
\end{lemma}
\begin{proof}
	Fix \(\epsilon>0\). Take \(S\) a finite subset of \(\bbS^{d-1}\) such that \(|S|\leq c\delta^{-d+1}\) and \(\bigcup_{s\in S} \cone_{s,\delta}\cap \calU_{\tilde{\nu}} = \calU_{\tilde{\nu}}\). For \(s\in S\), denote \(A_s = \partial^{\exterior} (L\calU_{\tilde{\nu}}) \cap \cone_{s,\delta}\). Then,
	\begin{align*}
	P\big(0\leftrightarrow (L\calU_{\tilde{\nu}})^c \big) &\leq \sum_{s\in S}\sum_{x\in A_s }P\big(0\xleftrightarrow[]{L\calU_{\tilde{\nu}}} x \big)\\
	&\leq \theta^{-c'\delta L}(c''L^{d-1}) \sum_{s\in S}P\big(0\xleftrightarrow[]{L\calU_{\tilde{\nu}}} \frac{sL}{\tilde{\nu}(s)} \big),
	\end{align*}
	where we used insertion tolerance in the second line. Now, for any fixed \(s\in S\), let \(s^{\star}\) be dual to \(s\). See Figure~\ref{fig:Convexity_connect_calU}. One then has
	\begin{equation*}
	P\big(0\xleftrightarrow[]{L\calU_{\tilde{\nu}}} \frac{sL}{\tilde{\nu}(s)} \big) \leq P\big( 0\leftrightarrow \frac{L\lrangle{s,s^{\star}}}{\tilde{\nu}(s)}s^{\star} + H_{s^{\star}} \big)
	\leq e^{-\frac{L\lrangle{s,s^{\star} }}{\tilde{\nu}(s) }\nu_H(s^{\star})(1+o_L(1)) } = e^{-L(1+o_L(1)) }.
	\end{equation*} Now, the \(o_L(1)\) depends on \(s\). Write it \(o_L^s(1)\). One therefore obtains
	\begin{equation*}
	P\big(0\leftrightarrow (L\calU_{\tilde{\nu}})^c \big) \leq \theta^{-c\delta L}(c'L^{d-1}) c''\delta^{-d+1} e^{-L }e^{\max_{s\in S}o_L^s(1)}.
	\end{equation*}
	Take \(\delta\) small enough and then \(L\) large enough to have \(\theta^{-c\delta L}(c'L^{d-1}) c''\delta^{-d+1} \leq e^{\epsilon L/2}\) and \(\max_{s\in S}o_L^s(1)\leq \epsilon L/2\).
	\begin{figure}[h]
		\centering
		\includegraphics[scale=0.7]{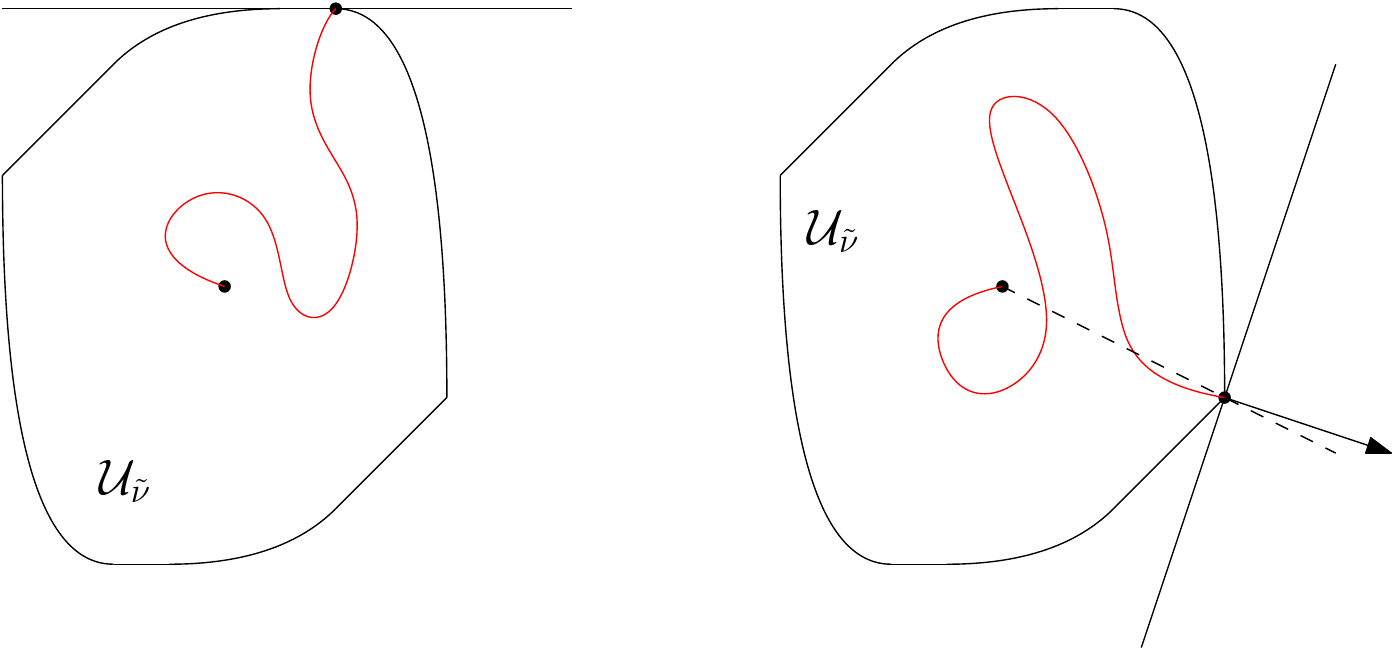}
		\caption{For each direction \(s\), we chose a dual direction for which connecting to a half-spaces is the same as connecting in direction \(s\).}
		\label{fig:Convexity_connect_calU}
	\end{figure}
\end{proof}

We then use the coarse graining procedure of Section~\ref{sec:coarse_graining} with \(\Delta = L\calU_{\tilde{\nu}}\) and \(K=\log(L)^2\): \(\CG_L\equiv \CG_{L\calU_{\tilde{\nu}}, \log(L)^2}\).

As a corollary of this construction, we obtain
\begin{corollary}
	For any \(s\in\bbS^{d-1}\),
	\begin{equation*}
		\overline{\nu}(s)\leq \tilde{\nu}(s) \leq \underline{\nu}(s).
	\end{equation*}In particular, \(\nu\) is well defined and defines a norm on \(\Rd\).
\end{corollary}
\begin{proof}
	Fix some \(s\in\bbS^{d-1}\). One has the direct lower bound \(\tilde{\nu}(s) \geq \overline{\nu}(s)\). To obtain the other bound, we use \(\CG_L\). Any cluster \(C\ni 0,Ns\) has \(|f(C)|\geq\frac{N\tilde{\nu}(s)}{L+\log(L)^2}\) (recall \(\calU_{\tilde{\nu}}\) is convex). Fix \(\epsilon>0\) small and take \(L\geq L_0(\epsilon)\). Using the bound on the combinatoric of trees and Lemmas~\ref{lem:tree_energy_bound} and~\ref{lem:main_energy}, one obtains
	\begin{equation*}
		P(0\leftrightarrow Ns) \leq e^{(\epsilon + o_L(1)) N} e^{-N\tilde{\nu}(s)}.
	\end{equation*}Taking the log, dividing by \(-N\) and letting \(N\to\infty\) gives \(\underline{\nu}(s)\geq \tilde{\nu}(s) -\epsilon + o_L(1)\). Letting \(L\to\infty\) and then \(\epsilon\searrow 0\) give the result.
\end{proof}

\section*{Acknowledgements}

The author thanks the university Roma Tre for its hospitality and is supported by the Swiss NSF through an early PostDoc.Mobility Grant. The author also thanks Yvan Velenik for a careful reading of the manuscript and for useful discussions.

\appendix

\section{Relaxed Fekete's Lemma}

We use this Lemma which proof is an easy adaptation of the usual Fekete's Lemma.
\begin{lemma}
	\label{app:quasi_sub_add}
	Suppose \((a_n)_{n\geq 1}\) is a sequence with \(c_-n<a_n<c_+n\) for some \(0<c_-\leq c_+ <\infty\). Suppose that there exists \(N_0\geq 1\) and functions \(f,g:(\bbZ_{>0})\to\bbZ\) such that
	\begin{itemize}
		\item \(f(n) = o(n)\), \(g(n) = o(n)\),
		\item For any \(n,m\geq N_0\), \(a_{n+m+g(\min(n,m))}\leq a_n+a_m+f(\min(n,m))\).
	\end{itemize}Then, the limit \(\lim_{n\to\infty} \frac{a_n}{n}\) exists in \([c_-,c_+]\).
\end{lemma}
\begin{proof}
	Let \(\underline{l} = \liminf_{n\to\infty}\frac{a_n}{n} \). Let \((n_k)_{k\geq 1}\) be an increasing sequence such that \(\lim_{k\to\infty} \frac{a_{n_k}}{n_k} =\underline{l} \). Fix \(k\) such that \(n_k\geq N_0\). For any \(N\) large enough, \(N= q (n_k+g(n_k)) + r\) with \(r<n_k+g(n_k)\). Then, by \(q-1\) iterations of our sub-additivity-type hypotheses
	\begin{equation*}
		\frac{a_N}{N} \leq \frac{(q-1)(a_{n_k} + f(n_k)) + a_{n_k+g(n_k) + r}}{q (n_k+ g(n_k)) + r} = \underline{l} + o_k(1)+o_{n_k}(1)+o_N(1).
	\end{equation*}Taking \(N\to\infty\), one obtains
	\begin{equation*}
		\limsup_{N\to\infty}\frac{a_N}{N} \leq \underline{l} + o_k(1)+o_{n_k}(1).
	\end{equation*}\(k\) being arbitrary, one can now take \(k\to\infty\) to obtain the wanted result.
\end{proof}

\bibliographystyle{plain}
\bibliography{BIGbib}

\end{document}